\pdfoutput=1
\documentclass[bj]{imsart}

\RequirePackage{amsthm,amsmath,amsfonts,amssymb}
\RequirePackage[numbers]{natbib}
\RequirePackage{hyperref}
\hypersetup{colorlinks,citecolor=blue,urlcolor=blue,dvipdfm}
\RequirePackage[dvipdfmx]{graphicx}
\RequirePackage{mathtools}
\RequirePackage{wrapfig}
\RequirePackage{algorithm}
\RequirePackage{algorithmic}
\mathtoolsset{showonlyrefs=true}
\startlocaldefs

\newtheorem{Theorem}{Theorem}[section]
\newtheorem{Lemma}[Theorem]{Lemma}
\newtheorem{Corollary}{Corollary}[Theorem]
\theoremstyle{remark}
\newtheorem{Definition}[Theorem]{Definition}

\endlocaldefs

\newcommand{\Ent}{\mathrm{Ent}}

\newcommand{\sm}{\scalebox{0.6}{\rm T}}

\begin{document}

\begin{frontmatter}
\title{Entropy-regularized optimal transport on multivariate normal and $q$-normal distributions}
\runtitle{}

\begin{aug}
\author[A]{\fnms{QIJUN} \snm{TONG}\ead[label=e1]{tong.qujin@gmail.com}}
\and
\author[B]{\fnms{KEI} \snm{KOBAYASHI}\ead[label=e2]{kei@math.keio.ac.jp}}
\address[A]{Department of Mathematics, Faculty of Science and Technology,
   Keio University, Yokohama, Japan.
\printead{e1}}
\address[B]{Department of Mathematics, Faculty of Science and Technology,
   Keio University, Yokohama, Japan.
\printead{e2}}
\end{aug}
\begin{abstract}
The distance and divergence of the probability measures play a central role in statistics, machine learning, and many other related fields.
The Wasserstein distance has received much attention in recent years because of its distinctions from other distances or divergences.
Although~computing the Wasserstein distance is costly, entropy-regularized optimal transport was proposed to computationally efficiently approximate the Wasserstein distance. The purpose of this study is to understand the theoretical aspect of entropy-regularized optimal transport.
In this paper, we~focus on entropy-regularized optimal transport on multivariate normal distributions and \mbox{$q$-normal} distributions. We~obtain the explicit form of the entropy-regularized optimal transport cost on multivariate normal and $q$-normal distributions; this provides a perspective to understand the effect of entropy regularization, which was previously known only experimentally. 
Furthermore, we obtain the entropy-regularized Kantorovich estimator for the probability measure that satisfies certain conditions. We also demonstrate how the Wasserstein distance, optimal coupling, geometric structure, and statistical efficiency are affected by entropy regularization in some experiments.
In particular, our results about the explicit form of the optimal coupling of the Tsallis entropy-regularized optimal transport on multivariate $q$-normal distributions and the entropy-regularized Kantorovich estimator are novel and will become the first step towards the understanding of a more general setting.
\end{abstract}

\end{frontmatter}

\section{Introduction}
Comparing probability measures is a fundamental problem in statistics and machine learning. A classical way to compare probability measures is the Kullback--Leibler divergence. Let $M$ be a measurable space and $\mu,\nu$ be the probability measure on $M$; then,~the~Kullback--Leibler divergence is defined as:
\begin{equation}
\mathrm{KL}(\mu|\nu)=\int_{M} d\mu\log \frac{d\mu}{d\nu}.
\end{equation}

The Wasserstein distance \cite{villani}, also known as the earth mover distance \cite{emd}, is another way of comparing probability measures. It is a metric on the space of probability measures derived by the mass transportation theory of two probability measures. Informally,~optimal~transport theory considers an optimal transport plan between two probability measures under a cost function, and the Wasserstein distance is defined by the minimum total transport cost. A significant difference between the Wasserstein distance and the Kullback--Leibler divergence is that the former can reflect the metric structure, whereas the latter cannot. The Wasserstein distance can be written as:
\begin{equation}
W_{p}(\mu,\nu):=\left\{\inf_{\pi\in\Pi(\mu,\nu)}\int_{M\times M}d(x,y)^{p}d\pi(x,y)\right\}^{\frac{1}{p}},
\end{equation}
where $d(\cdot,\cdot)$ is a distance function on a measurable metric space $M$ and $\Pi(\mu,\nu)$ denotes the set of probability measures on $M\times M$, whose marginal measures correspond to $\mu$ and $\nu$.
In recent years, the application of optimal transport and the Wasserstein distance has been studied in many fields such as statistics, machine learning, and image processing. For example, Reference \cite{conv} generated the interpolation of various three-dimensional (3D) objects using the Wasserstein barycenter. In the field of word embedding in natural language processing, Reference \cite{elipap} embedded each word as an elliptical distribution, and~the Wasserstein distance was applied between the elliptical distributions. There are many studies on the applications of optimal transport to deep learning, including \cite{inception,wgan,nitanda}. Moreover, Reference \cite{sonoda} analyzed the denoising autoencoder \cite{dae} with gradient flow in the Wasserstein space.

In the application of the Wasserstein distance, it is often considered in a discrete setting where $\mu$ and $\nu$ are discrete probability measures. Then, obtaining the Wasserstein distance between $\mu$ and $\nu$ can be formulated as a linear programming problem. In general, however, it is computationally intensive to solve such linear problems and obtain the optimal coupling of two probability measures.
For such a situation, a novel numerical method, entropy regularization, was proposed by \cite{cuturi},
\begin{equation}
C_{\lambda}(\mu,\nu):=\inf_{\pi\in\Pi(\mu,\nu)}\int_{\mathbb{R}^{n}\times \mathbb{R}^{n}}c(x,y)\pi({x},{y})d{x}d{y}-\lambda \mathrm{Ent}(\pi).
\end{equation}

This is a relaxed formulation of the original optimal transport of a cost function $c(\cdot,\cdot)$, in which the negative Shannon entropy $-\Ent(\cdot)$ is used as a regularizer. For a small $\lambda$, $C_{\lambda}(\mu,\nu)$ can approximate the $p$-th power of the Wasserstein distance between two discrete probability measures, and it can be computed efficiently by using Sinkhorn's algorithm~\cite{sink}. 

More recently, many studies have been published on improving the computational efficiency.
According to \cite{lin1}, the most computationally efficient algorithm at this moment to solve the linear problem for the Wasserstein distance is Lee--Sidford linear solver \cite{linear_prog}, which~runs in $O(n^{2.5})$. Reference \cite{dvurechensky} proved that a complexity bound for the Sinkhorn algorithm is $\tilde{O}(n^2\varepsilon^{-2})$, where $\varepsilon$ is the desired absolute performance guarantee. After~\cite{cuturi} appeared, various algorithms have been proposed. For example, Reference \cite{stosink} adopted stochastic optimization schemes for solving the optimal transport. The Greenkhorn algorithm \cite{greedhorn} is the greedy variant of the Sinkhorn algorithm, and Reference \cite{lin1} proposed its acceleration. Many other approaches such as adapting a variety of standard optimization algorithms to approximate the optimal transport problem can be found in~\cite{blondelsmooth,cuturismoothed,lin1,lin2}. Several~specialized Newton-type algorithms \cite{much,cohen2017matrix} achieve complexity bound $\tilde{O}(n^2\varepsilon^{-1})$~ \cite{blanchet2018towards,quanrud2018approximating}, which are the best ones in terms of computational complexity at the present moment.

Moreover, entropy-regularized optimal transport has another advantage. Because of the differentiability of the entropy-regularized optimal transport and the simple structure of Sinkhorn's algorithm, we can easily compute the gradient of the entropy-regularized optimal transport cost and optimize the parameter of a parametrized probability distribution by using numerical differentiation or automatic differentiation. Then, we can define a differentiable loss function that can be applied to various supervised learning methods \cite{wloss}. Entropy-regularized optimal transport can be used to approximate not only the Wasserstein distance, but also its optimal coupling as a mapping function. Reference \cite{domain} adopted the optimal coupling of the entropy-regularized optimal transport as a mapping function from one domain to another.

Despite the empirical success of the entropy-regularized optimal transport, its theoretical aspect is less understood. Reference \cite{beru} studied the expected Wasserstein distance between a probability measure and its empirical version. Similarly, Reference \cite{dual} showed the consistency of the entropy-regularized optimal transport cost between two empirical distributions. Reference \cite{mld} showed that minimizing the entropy-regularized optimal transport cost between empirical distributions is equivalent to a type of maximum likelihood estimator. Reference \cite{VAE} considered Wasserstein generative adversarial networks with an entropy regularization.
Reference \cite{amari_was} constructed information geometry from the convexity of the entropy-regularized optimal transport cost.

Our intrinsic motivation of this study is to produce an analytical solution about the entropy-regularized optimal transport problem between continuous probability measures so that we can gain insight into the effects of entropy regularization in a theoretical, as~ well as an experimental way. In our study, we generalized the Wasserstein distance between two multivariate normal distributions by entropy regularization. We derived the explicit form of the entropy-regularized optimal transport cost and its optimal coupling, which can be used to analyze the effect of entropy regularization directly.
In general, the~nonregularized Wasserstein distance between two probability measures and its optimal coupling cannot be expressed in a closed form; however, Reference \cite{frec} proved the explicit formula for multivariate normal distributions. Theorem \ref{main} is a generalized form of \cite{frec}. We obtain an explicit form of the entropy-regularized optimal transport between two multivariate normal distributions. Furthermore, by adopting the Tsallis entropy \cite{tsallis} as the entropy regularization instead of the Shannon entropy, our theorem can be generalized to multivariate $q$-normal distributions.


Some readers may find it strange to study the entropy-regularized optimal transport for multivariate normal distributions,
where the exact (nonregularized) optimal transport has been obtained explicitly.
However, we think it is worth studying from several perspectives:
\begin{itemize}
\item Normal distributions are the simplest and best-studied probability distributions, and thus, it is useful to examine the regularization theoretically in order to infer results for other distributions.
In particular, we will partly answer the questions ``How much do entropy constraints affect the results?'' and ``What does it mean to constrain by the entropy?'’ for the simplest cases.
Furthermore, as a first step in constructing a theory for more general probability distributions, in Section \ref{Extension to Tsallis entropy regularization}, we propose a generalization to multivariate $q$-normal distributions.
\item Because normal distributions are the limit distributions in asymptotic theories using the central limit theorem, studying normal distributions is necessary for the asymptotic theory of regularized Wasserstein distances and estimators computed by them. Moreover, it was proposed to use the entropy-regularized Wasserstein distance to compute a lower bound of the generalization error for a variational autoencoder~\cite{VAE}. The~study of the asymptotic behavior of such bounds is one of the expected applications of our results.
\item Though this has not yet been proven theoretically, we suspect that entropy regularization is efficient not only for computational reasons, such as the use of the Sinkhorn algorithm, but also in the sense of efficiency in statistical inference. Such a phenomenon can be found in some existing studies, including \cite{amari_new}. Such statistical efficiency is confirmed by some experiments in Section \ref{Numerical Experiments}.
\end{itemize}

The remainder of this paper is organized as follows. First, we review some definitions of optimal transport and entropy regularization in Section \ref{Preliminary}. Then, in Section \ref{Entropy-Regularized Optimal Transport between Multivariate Normal Distributions}, we provide an explicit form of the entropy-regularized optimal transport cost and its optimal coupling between two multivariate normal distributions. We also extend this result to $q$-normal distributions for Tsallis entropy regularization in Section \ref{Extension to Tsallis entropy regularization}.
In Section \ref{Entropy-Regularized Kantorovich Estimator}, we obtain the entropy-regularized Kantorovich estimator of probability measures on $\mathbb{R}^{n}$ with a finite second moment that are absolutely continuous with respect to the Lebesgue measure in Theorem \ref{thm:mke}. We emphasize that Theorem \ref{thm:mke} is not limited to the case of multivariate normal distribution, but can handle a wider range of probability measures.
We analyze how entropy regularization affects the optimal result experimentally in certain sections.

We note that after publishing the preprint version of the paper, we found closely related results \cite{janati,malla} reported within half a year. In Janati et al. \cite{janati}, they proved the same result as Theorem \ref{main} based on solving the fixed-point equation behind Sinkhorn's algorithm. Their results include the unbalanced optimal transport between unbalanced multivariate normal distributions. They also studied the convexity and differentiability of the objective function of the entropy-regularized optimal transport. In \cite{malla}, the same closed-form as Theorem \ref{main} was proven by ingeniously using the Schr\"{o}dinger system.
Although there are some overlaps, our paper has significant novelty in the following respects. Our proof is more direct than theirs and can be extended directly to the proof for the Tsallis entropy-regularized optimal transport between multivariate $q$-normal distributions provided in Section \ref{Extension to Tsallis entropy regularization}. Furthermore, Corollaries \ref{cor:mono} and \ref{cor:1} are novel and important results to evaluate how much the entropy regularization affects the estimation results or not at all. We also obtain the entropy-regularized Kantorovich estimator in Theorem \ref{thm:mke}. 

\section{Preliminary}
\label{Preliminary}
In this section, we review some definitions of optimal transport and entropy-regularized optimal transport. These definitions were referred to in \cite{copcuturi,villani}.
In this section, we use a tuple $(M,\Sigma)$ for a set $M$ and $\sigma$-algebra on $M$ and $\mathcal{P}(X)$ for the set of all probability measures on a measurable space $X$.

\begin{Definition}[Pushforward measure]
Given measurable spaces $(M_{1},\Sigma_{1})$ and $(M_{2},\Sigma_{2})$, a~measure $\mu:\Sigma_{1}\rightarrow[0,+\infty]$, and a measurable mapping $\varphi: M_{1} \rightarrow M_{2}$, the pushforward measure of $\mu$ by $\varphi$ is defined by:
\begin{equation}
\forall B \in \Sigma_{2},\ \varphi _ { \# } \mu ( B ) : = \mu \left( \varphi ^ { - 1 } ( B ) \right) .
\end{equation}
\end{Definition}

\begin{Definition}[Optimal transport map]
Consider a measurable space $(M,\Sigma)$, and let $c:M\times M\rightarrow\mathbb{R}_{+}$ denote a cost function. Given $\mu,\nu\in\mathcal{P}(M)$, we call $\varphi:M\rightarrow M$ the optimal transport map if $\varphi$ realizes the infimum of:
\begin{equation}
\inf_{\varphi_{\#}\mu=\nu}\int _ {M} c ( x , \varphi ( x ) ) d \mu ( x ).
\end{equation}
\end{Definition}
This problem was originally formalized by \cite{monge}. However, the optimal transport map does not always exist. Then, Kantorovich considered a relaxation of this problem in \cite{kantorovich}.

\begin{Definition}[Coupling]
Given $\mu,\nu\in\mathcal{P}(M)$, the coupling of $\mu$ and $\nu$ is a probability measure on $M\times M$ that satisfies:

\begin{equation}
\forall A \in \Sigma ,\ \pi ( A \times M ) = \mu ( A ) ,\quad \pi ( M \times A ) = \nu ( A ).
\end{equation}

\end{Definition}
\begin{Definition}[Kantorovich problem]
The Kantorovich problem is defined as finding a coupling $\pi$ of $\mu$ and $\nu$ that realizes the infimum of:
\begin{equation}
\int _ { M \times M } c ( x , y ) d \pi ( x , y ).
\end{equation}
\end{Definition}
Hereafter, let $\Pi(\mu,\nu)$ be the set of all couplings of $\mu$ and $\nu$. 
When we adopt a distance function as the cost function, we can define the Wasserstein distance.
\begin{Definition}[Wasserstein distance]

Given $p\geq 1$, a measurable metric space $(M,\Sigma,d)$, and~$\mu,\nu\in \mathcal{P}(M)$ with a finite $p$-th moment, the $p$-Wasserstein distance between $\mu$ and $\nu$ is defined as:
\begin{equation}
W_{p}(\mu,\nu):=\inf_{\pi\in\Pi(\mu,\nu)}\left(\int_{M\times M}d(x,y)^{p}d\pi(x,y)\right)^{\frac{1}{p}}.
\end{equation}

\end{Definition}
Now, we review the definition of entropy-regularized optimal transport on $\mathbb{R}^n$.
\begin{Definition}[Entropy-regularized optimal transport]
Let $\mu,\nu\in \mathcal{P}(\mathbb{R}^{n})$, $\lambda>0$, and let $\pi(x,y)$ be the density function of the coupling of $\mu$ and $\nu$, whose reference measure is the Lebesgue measure. We define the entropy-regularized optimal transport cost as:
\begin{equation}
C_{\lambda}(\mu,\nu):=\inf_{\pi\in\Pi(\mu,\nu)}\int_{\mathbb{R}^{n}\times \mathbb{R}^{n}}c(x,y)\pi({x},{y})d{x}d{y}-\lambda \mathrm{Ent}(\pi),
\end{equation}
where $\Ent(\cdot)$ denotes the Shannon entropy of a probability measure:
\begin{equation}
\mathrm{Ent}(\pi)=-\int_{\mathbb{R}^{n}\times\mathbb{R}^{n}}\pi(x,y)\log\pi(x,y)dxdy.
\end{equation}
\end{Definition}
There is another variation in entropy-regularized optimal transport defined by the relative entropy instead of the Shannon entropy:
\begin{equation}
\tilde{C}_{\lambda}(\mu,\nu):=\inf_{\pi\in\Pi(\mu,\nu)}\int_{\mathbb{R}^{n}\times \mathbb{R}^{n}}c(x,y)\pi({x},{y})d{x}d{y}+\lambda \mathrm{KL}(\pi | d\mu\otimes d\nu).
\end{equation}

This is definable even when $ \Pi(\mu,\nu)$ includes a coupling that is not absolutely continuous with respect to the Lebesgue measure.
We note that when both $\mu$ and $\nu$ are absolutely continuous, the infimum is attained by the same $\pi$ for $C_{\lambda}$ and $\tilde{C}_{\lambda}$, and it depends only on $\mu$ and $\nu$. In the following part of the paper, we assume the absolute continuity of $\mu,\nu$, and~$\pi$ with respect to the Lebesgue measure for well-defined entropy regularization.
\section{Entropy-Regularized Optimal Transport between Multivariate Normal Distributions}
\label{Entropy-Regularized Optimal Transport between Multivariate Normal Distributions}
In this section, we provide a rigorous solution of entropy-regularized optimal transport between two multivariate normal distributions.
Throughout this section, we adopt the squared Euclidean distance $\|x-y\|^{2}$ as the cost function. To prove our theorem, we start by expressing $C_{\lambda}$ using mean vectors and covariance matrices. The following lemma is a known result; for example, see \cite{frec}.

\begin{Lemma}\label{lma:rvtrans}
Let $X\sim P, Y\sim Q$ be two random variables on $\mathbb{R}^{n}$ with means $\mu_{1},\mu_{2}$ and covariance matrices $\Sigma_{1},\Sigma_{2}$, respectively. If $\pi(x,y)$ is a coupling of $P$ and $Q$, we have:
\begin{equation}
\int_{\mathbb{R}^{n}\times \mathbb{R}^{n}}\|{x}-{y}\|^2\pi({x},{y})dxdy=\|\mu_{1}-\mu_{2}\|^{2}+\mathrm{tr}\left\{\Sigma_{1}+\Sigma_{2}-2\mathrm{Cov}(X,Y)\right\}.\label{siki:rvtrans}
\end{equation}
\end{Lemma}
\begin{proof}
Without loss of generality, we can assume $X$ and $Y$ are centralized, because:
\begin{equation}
\int\|(x-\mu_{1})-(y-\mu_{2})\|^{2}\pi(x,y)dxdy=\int\|x-y\|^{2}\pi(x,y)dxdy-\|\mu_{1}-\mu_{2}\|^{2}.
\end{equation}

Therefore, we have:
\begin{align}
\int\|{x}-{y}\|^2\pi({x},{y})d{x}d{y}&=E[\|X-Y\|^2]=E[\mathrm{tr}\{(X-Y)(X-Y)^{\sm}\}]\nonumber\\
&=\mathrm{tr}\left\{\Sigma_{1}+\Sigma_{2}-2\mathrm{Cov}(X,Y)\right\}\label{rvtrans}.
\end{align}
By adding $\|\mu_{1}-\mu_{2}\|^{2}$, we obtain \eqref{siki:rvtrans}.
\end{proof}

Lemma \ref{lma:rvtrans} shows that $\int_{\mathbb{R}^{n}\times \mathbb{R}^{n}}\|{x}-{y}\|^2\pi({x},{y})d{x}d{y}$ can be parameterized by the covariance matrices $\Sigma_{1},\Sigma_{2},\mathrm{Cov}(X,Y)$. Because $\Sigma_{1}$ and $\Sigma_{2}$ are fixed, the infinite-dimensional optimization of the coupling $\pi$ is a finite-dimensional optimization of covariance matrix $\mathrm{Cov}(X,Y)$.

We prepare the following lemma to prove Theorem \ref{main}. 
\begin{Lemma}
Under a fixed mean and covariance matrix, the probability measure that maximizes the entropy is a multivariate normal distribution.
\label{maxent}
\end{Lemma}
Lemma \ref{maxent} is a particular case of the principle of maximum entropy \cite{jaynes}, and the proof can be found in \cite{mep} Theorem 3.1.

\begin{Theorem}\label{main}

Let $P\sim\mathcal{N}({\mu_1},\Sigma_1),Q\sim\mathcal{N}({\mu_2},\Sigma_2)$ be two multivariate normal distributions. The~optimal coupling $\pi$ of $P$ and $Q$ of the entropy-regularized optimal transport:
\begin{equation}
C_{\lambda}(P,Q)=\inf_{\pi\in\Pi(P,Q)}\int_{\mathbb{R}^{n}\times \mathbb{R}^{n}}\|{x}-{y}\|^2\pi({x},{y})d{x}d{y}-4\lambda \mathrm{Ent}(\pi). \quad\quad\quad(\text{\textasteriskcentered})
\label{siki:frac2}
\end{equation}
is expressed as:
\begin{equation}
\pi\sim\mathcal{N}\left(\begin{pmatrix}
{\mu}_1\\
{\mu}_2 \\
\end{pmatrix},
\begin{pmatrix}
\Sigma_1 &\Sigma_\lambda \\
\Sigma_\lambda^{\sm} & \Sigma_2 \\
\end{pmatrix}\right)
\end{equation}
where:
\begin{equation}
\Sigma_\lambda:=\Sigma_1^{1/2}(\Sigma_1^{1/2}\Sigma_2\Sigma_1^{1/2}+\lambda^2I)^{1/2}\Sigma_1^{-1/2}-{\lambda}I.
\label{optsigma}
\end{equation}

Furthermore, $C_{\lambda}(P,Q)$ can be written as:
\begingroup\makeatletter\def\f@size{9}\check@mathfonts
\def\maketag@@@#1{\hbox{\m@th\normalsize\normalfont#1}}%
\begin{align}
C_{\lambda}(P,Q)=&\|\mu_{1}-\mu_{2}\|^{2}+\mathrm{tr}(\Sigma_{1}+\Sigma_{2}-2(\Sigma_{1}^{1/2}\Sigma_{2}\Sigma_{1}^{1/2}+\lambda^{2}I)^{1/2})\nonumber\\
-&2{\lambda}\log |(\Sigma_{1}^{1/2}\Sigma_{2}\Sigma_{1}^{1/2}+\lambda^{2}I)^{1/2}-\lambda I|-2\lambda n\log(2\pi\lambda)-4\lambda n\log(2\pi)-2\lambda n
\label{dist}
\end{align}
\endgroup
and the relative entropy version can be written as:
\begin{equation}
\tilde{C}_{\lambda}(P,Q)=C_{\lambda}(P,Q)+2{\lambda}\log |\Sigma_1|| \Sigma_2|+4\lambda n\{\log(2\pi)+1\}.
\end{equation}
\end{Theorem}
We note that we use the regularization parameter $4\lambda$ in $(\text{\textasteriskcentered})$ for the sake of simplicity.
\begin{proof}

Although the first half of the proof can be derived directly from Lemma \ref{maxent}, we~provide a proof of this theorem by Lagrange calculus, which will be used later for the extension to $q$-normal distributions. Now, we define an optimization problem that is equivalent to the entropy-regularized optimal transport as follows:
\begin{align}
\mathrm{minimize}\ &\int\|x-y\|^{2}\pi(x,y)dxdy-4\lambda\mathrm{Ent}(\pi) \\
\mathrm{subject\ to}\ \ &\int\pi(x,y)dx=q(y)\ \text{for}\ \forall\ y\in\mathbb{R}^n,\nonumber\\
&\int\pi(x,y)dy=p(x)\ \text{for}\ \forall\ x\in\mathbb{R}^n\ . 
\label{siki:constraint}
\end{align}

Here, $p(x)$ and $q(y)$ are probability density functions of $P$ and $Q$, respectively. 
Let~$\alpha(x)$, $\beta(y)$ be Lagrange multipliers that correspond to the above two constraints. The Lagrangian function of \eqref{siki:constraint} is defined as:
\begin{align}
L(\pi,\alpha,\beta):&=\int\|x-y\|^2\pi(x,y)dxdy+4\lambda\int\pi(x,y)\log\pi(x,y)dxdy\nonumber\\
&-\int\alpha(x)\pi(x,y)dxdy+\int\alpha(x)p(x)dx\nonumber\\
&-\int\beta(y)\pi(x,y)dxdy+\int\beta(y)q(y)dy.\label{laggene}
\end{align}

Taking the functional derivative of \eqref{laggene} with respect to $\pi$, we obtain:
\begin{equation}
\delta L(\pi,\alpha,\beta)=\int\left(\|x-y\|^{2}+4\lambda\log\pi(x,y)-\alpha(x)-\beta(y)\right)\delta\pi(x,y)dxdy.
\end{equation}

By the fundamental lemma of the calculus of variations, we have:
\begin{equation}
\pi({x},{y})\propto\exp\left(\alpha({x})+\beta({y})-\frac{\|{x}-{y}\|^2}{4\lambda}\right). \label{siki:lagmoto}
\end{equation}

Here, $\alpha(x),\beta(y)$ are determined from the constraints \eqref{siki:constraint}.
We can assume that $\pi$ is a $2n$-variate normal distribution, because for a fixed covariance matrix $\mathrm{Cov}(X,Y)$, $-\Ent(\pi)$ takes the infimum when the coupling $\pi$ is a multivariate normal distribution by Lemma \ref{maxent}. 
Therefore, we can express $\pi$ by using $z=(x^{\sm},y^{\sm})^{\sm}$ and a covariance matrix \mbox{$\Sigma:=\mathrm{Cov}(X,Y)$}~as: 
\begin{equation}\label{mulnorm}
\pi(x,y)\propto\exp\left\{-\frac{1}{2}z^{\sm}
\begin{pmatrix}
\Sigma_{1}&\Sigma_{\phantom{2}} \\
\Sigma^{\sm} & \Sigma_{2} \\
\end{pmatrix}^{-1}
z\right\}.
\end{equation}

Putting: 
\begin{equation}
\begin{pmatrix}
\tilde{\Sigma}_{1}&\tilde{\Sigma} \\
\tilde{\Sigma}^{\sm} & \tilde{\Sigma}_{2} \\
\end{pmatrix}:=
\begin{pmatrix}
\Sigma_{1}&\Sigma_{\phantom{2}} \\
\Sigma^{\sm} & \Sigma_{2} \\
\end{pmatrix}^{-1}
,
\end{equation}
we write:
\begin{align}
-\frac{1}{2}z^{\sm}\begin{pmatrix}
\Sigma_{1}&\Sigma_{\phantom{2}} \\
\Sigma^{\sm} & \Sigma_{2} \\
\end{pmatrix}^{-1}z&=-\frac12\begin{pmatrix} x^{\sm}&y^{\sm} \end{pmatrix}
\begin{pmatrix}
\tilde{\Sigma}_{1}&\tilde{\Sigma} \\
\tilde{\Sigma}^{\sm} & \tilde{\Sigma}_{2} \\
\end{pmatrix}
\begin{pmatrix} x \\ y \end{pmatrix}\\
&=-\frac12 x^{\sm}\tilde{\Sigma}_{1}x-\frac12 y^{\sm}\tilde{\Sigma}_{2}y
-x^{\sm}\tilde{\Sigma}y.\label{siki:tenkai}
\end{align}

According to block matrix inversion formula \cite{cook},
$\tilde{\Sigma}=-\Sigma_1^{-1}\Sigma A^{-1}$ holds,
where~$A:=\Sigma_{2}-\Sigma^{\sm}\Sigma_{1}^{-1}\Sigma$ is positive definite.
Then, comparing the term $x^{\sm}y$ between $\eqref{siki:lagmoto}$ and $\eqref{siki:tenkai}$, we~obtain $\Sigma_1^{-1}\Sigma A^{-1}=\frac{1}{2\lambda}I$
and: 
\begin{equation}
2\lambda\Sigma_{1}^{-1}\Sigma=A=\Sigma_{2}-\Sigma^{\sm}\Sigma_{1}^{-1}\Sigma.
\label{eq:sym}
\end{equation}

Here, $\Sigma_{1}^{-1}\Sigma=\Sigma^{\sm}\Sigma_{1}^{-1}$ holds, because $A$ is a symmetric matrix, and thus, we obtain:
\begin{equation}
\lambda \Sigma_{1}^{-1}\Sigma+\lambda\Sigma^{\sm}\Sigma_{1}^{-1}=\Sigma_{2}-\Sigma^{\sm}\Sigma_{1}^{-1}\Sigma.
\end{equation}

Completing the square of the above equation, we obtain:
\begin{equation}
(\Sigma_{1}^{-1/2}(\Sigma+\lambda I)\Sigma_{1}^{1/2})^{\sm}(\Sigma_{1}^{-1/2}(\Sigma+\lambda I)\Sigma_{1}^{1/2})=\Sigma_{1}^{1/2}\Sigma_{2}\Sigma_{1}^{1/2}+\lambda^{2}I \label{eq:quad}
\end{equation}

Let $Q$ be an orthogonal matrix; then, \eqref{eq:quad} can be solved as:
\begin{equation}
\Sigma_{1}^{-1/2}(\Sigma+\lambda I)\Sigma_{1}^{1/2}=Q(\Sigma_{1}^{1/2}\Sigma_{2}\Sigma_{1}^{1/2}+\lambda^{2}I)^{1/2}.\label{eq:orth}
\end{equation}

We rearrange the above equation as follows:
\begin{equation}
\Sigma_{1}^{1/2}(\Sigma_{1}^{-1}\Sigma)\Sigma_{1}^{1/2}+\lambda I=Q(\Sigma_{1}^{1/2}\Sigma_{2}\Sigma_{1}^{1/2}+\lambda^{2}I)^{1/2}.
\end{equation}

Because the left terms and $(\Sigma_{1}^{ 1/2}\Sigma_{2}\Sigma_{1}^{ 1/2}+\lambda^{2}I)^{1/2}$ are all symmetric positive definite, we can conclude that $Q$ is the identity matrix by the uniqueness of the polar decomposition. Finally, we obtain:
\begin{equation}
\Sigma=\Sigma_1^{1/2}(\Sigma_1^{1/2}\Sigma_{2}\Sigma_1^{1/2}+\lambda^2 I)^{1/2}\Sigma_{1}^{-1/2}-{\lambda}I =:\Sigma_{\lambda}.
\end{equation}

We obtain \eqref{dist} by the direct calculation of $C_\lambda$ using Lemma \ref{lma:rvtrans} with this $\Sigma_{\lambda}$.
\end{proof}
The following corollary helps us to understand the properties of $\Sigma_\lambda$.
\begin{Corollary}\label{cor:mono}
Let $\nu_{\lambda,1}\leq\nu_{\lambda,2}\leq\dot\leq\nu_{\lambda,n}$ be the eigenvalues of $ \Sigma_{\lambda}$; then, $\nu_{\lambda,i}$ monotonically decreases with $\lambda$ for any $i\in\{1,2,\dot,n\}$.
\end{Corollary}
\begin{proof}
Because $\Sigma_{1}^{-1/2}\Sigma_{\lambda}\Sigma_{1}^{1/2}=(\Sigma_1^{1/2}\Sigma_{2}\Sigma_1^{1/2}+\lambda^2 I)^{1/2}-{\lambda}I$ has the same eigenvalues as $\Sigma_{\lambda}$, if we let $\{\nu_{0, i}\}$ be the eigenvalues of $\Sigma_{1}^{1/2}\Sigma_{2}\Sigma_{1}^{1/2}$, $\nu_{\lambda,i}=\sqrt{\nu_{0,i}+\lambda^{2}}-\lambda$, which is a monotonically decreasing function of the regularization parameter $\lambda$.
\end{proof}
By the proof, for large $\lambda$, we can prove $\Sigma_1^{-1/2}\Sigma_{\lambda}\Sigma_1^{1/2}\simeq \frac{1}{2\lambda}\Sigma_1^{1/2}\Sigma_2\Sigma_1^{1/2}$ by diagonalization and $\nu_{\lambda,i}\simeq \frac{1}{2\lambda}\nu_{0,i}$.
Thus, $\Sigma_\lambda\simeq \frac{1}{2\lambda}\Sigma_1\Sigma_2$, and each element of $\Sigma_\lambda$ converges to zero
as $\lambda\rightarrow \infty$.

We show how entropy regularization behaves in two simple experiments. We calculate the entropy-regularized optimal transport cost $\mathcal{N}\left(\tiny\begin{pmatrix}0\\0\end{pmatrix},\begin{pmatrix}1&0\\0&1\end{pmatrix}\normalsize\right)$ and $\mathcal{N}\left(\tiny\begin{pmatrix}0\\0\end{pmatrix},\begin{pmatrix}2&-1\\-1&2\end{pmatrix}\normalsize\right)$ in the original version and the relative entropy version in Figure \ref{fig:wdist}. We separate the entropy-regularized optimal transport cost into the transport cost term and regularization term and display both of them.

\vspace{-6pt}
\begin{figure}[H]
 \includegraphics[width=14cm]{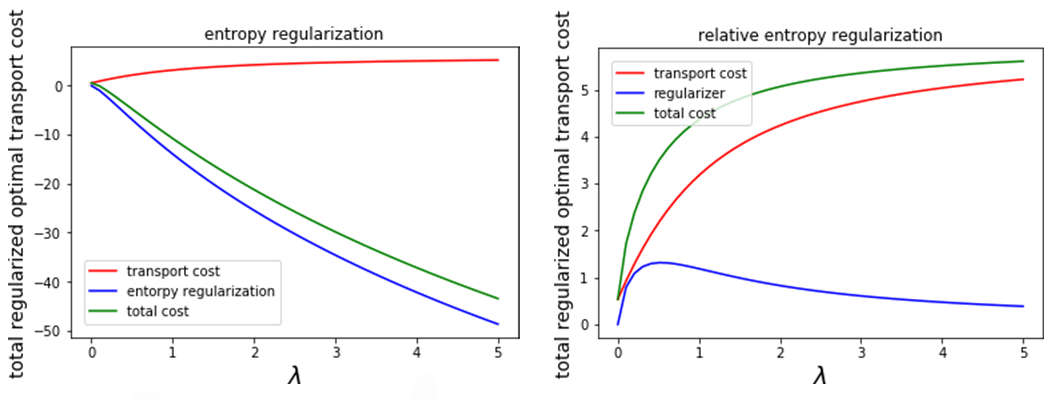}
 \caption[]{Graph of the entropy-regularized optimal transport cost between $\mathcal{N}\left(\tiny\begin{pmatrix}0\\0\end{pmatrix},\begin{pmatrix}1&0\\0&1\end{pmatrix}\normalsize\right)$ and $\mathcal{N}\left(\tiny\begin{pmatrix}0\\0\end{pmatrix},\begin{pmatrix}2&-1\\-1&2\end{pmatrix}\normalsize\right)$ with respect to $\lambda$ from zero to $10$. }
 \label{fig:wdist}
\end{figure}

It is reasonable that as $\lambda \downarrow 0$, $\Sigma_\lambda$ converges to $\Sigma_1^{1/2}(\Sigma_1^{1/2}\Sigma_2\Sigma_1^{1/2})^{1/2}\Sigma_1^{-1/2}$, which is equal to the original optimal coupling of nonregularized optimal transport and as $\lambda \rightarrow \infty$, $\Sigma_\lambda$ converges to {$\boldsymbol{0}$}. 
This is a special case of Corollary \ref{cor:mono}.
The larger $\lambda$ becomes, the less correlated the optimal coupling is. We visualize this behavior by computing the optimal couplings of two one-dimensional normal distributions in Figure \ref{fig:coup}.

\vspace{-6pt}
\nointerlineskip
\begin{figure}[H]
\begin{center}
 \includegraphics[width=15cm]{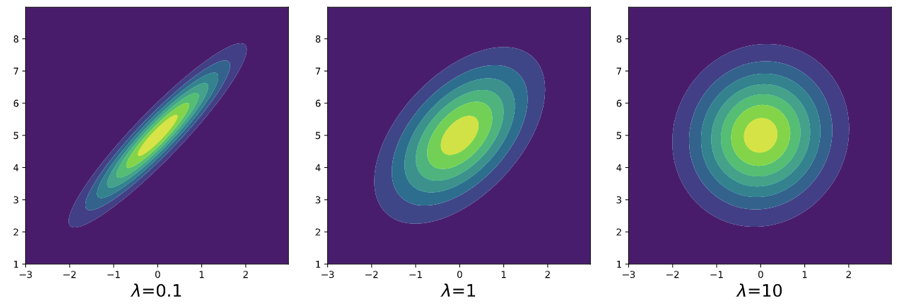}
\end{center}
\caption{Contours of the density functions of the entropy-regularized optimal coupling of $\mathcal{N}(0,1)$ and $\mathcal{N}(5,2)$ in three different parameters $\lambda=0.1,1,10$. All of the optimal couplings are two-variate normal distributions. \label{fig:coup}}
\end{figure}

The left panel shows the original version. The transport cost is always positive, and the entropy regularization term can take both signs in general; then, the sign and total cost depend on their balance. We note that the transport cost as a function of $\lambda$ is bounded, whereas the entropy regularization is not. The boundedness of the optimal cost is deduced from \eqref{lma:rvtrans} and Corollary \ref{cor:mono}, and the unboundedness of the entropy regularization is due to the regularization parameter $\lambda$ multiplied by the entropy. 
The right panel shows the relative entropy version. It always takes a non-negative value. Furthermore, because the total cost is bounded by the value for the independent joint distribution (which is always a feasible coupling), both the transport cost and the relative entropy regularization regularization term are also bounded. Nevertheless, the larger the regularization parameter $\lambda$, the greater the influence of entropy regularization over the total cost.

It is known that a specific Riemannian metric can be defined in the space of multivariate normal distributions, which induces the Wasserstein distance \cite{taka}. To understand the effect of entropy regularization, we illustrate how entropy regularization deforms this geometric structure in Figure \ref{fig:embd}. Here, we generate $100$ two-variate normal distributions $\{\mathcal{N}(0,\Sigma_{r,k})\}_{r,k\in\{1,2,\dot,10\}}$, where
$\{\Sigma_{r,k}\}$ is defined as:
\vspace{10pt}

\nointerlineskip
\begin{equation}
\quad \Sigma_{r,k}=
\begin{pmatrix}
\cos\left(2\pi\cdot\frac{k}{10}\right)&-\sin\left(2\pi\cdot\frac{k}{10}\right)\\
\sin\left(2\pi\cdot\frac{k}{10}\right)&\cos\left(2\pi\cdot\frac{k}{10}\right)
\end{pmatrix}^{\sm}
\begin{pmatrix}
1&0\\
0&\sqrt{\frac{r}{10}}
\end{pmatrix}
\begin{pmatrix}
\cos\left(2\pi\cdot\frac{k}{10}\right)&-\sin\left(2\pi\cdot\frac{k}{10}\right)\\
\sin\left(2\pi\cdot\frac{k}{10}\right)&\cos\left(2\pi\cdot\frac{k}{10}\right)
\end{pmatrix}.
\end{equation}


\vspace{7pt}

To visualize the geometric structure of these two-variate normal distributions, we compute the relative entropy-regularized optimal transport cost $\tilde{C}_{\lambda}$ between each pairwise two-variate normal distributions. Then, we apply multidimensional scaling \cite{kruskal1964nonmetric} to embed them into a plane (see Figure \ref{fig:embd}). We can see entropy regularization deforming the geometric structure of the space of multivariate normal distributions. The deformation for distributions close to the isotopic normal distribution is more sensitive to the change in $\lambda$.

\vspace{-3pt}

\nointerlineskip
\begin{figure}[H]
\begin{center}
 \includegraphics[width=15.5cm]{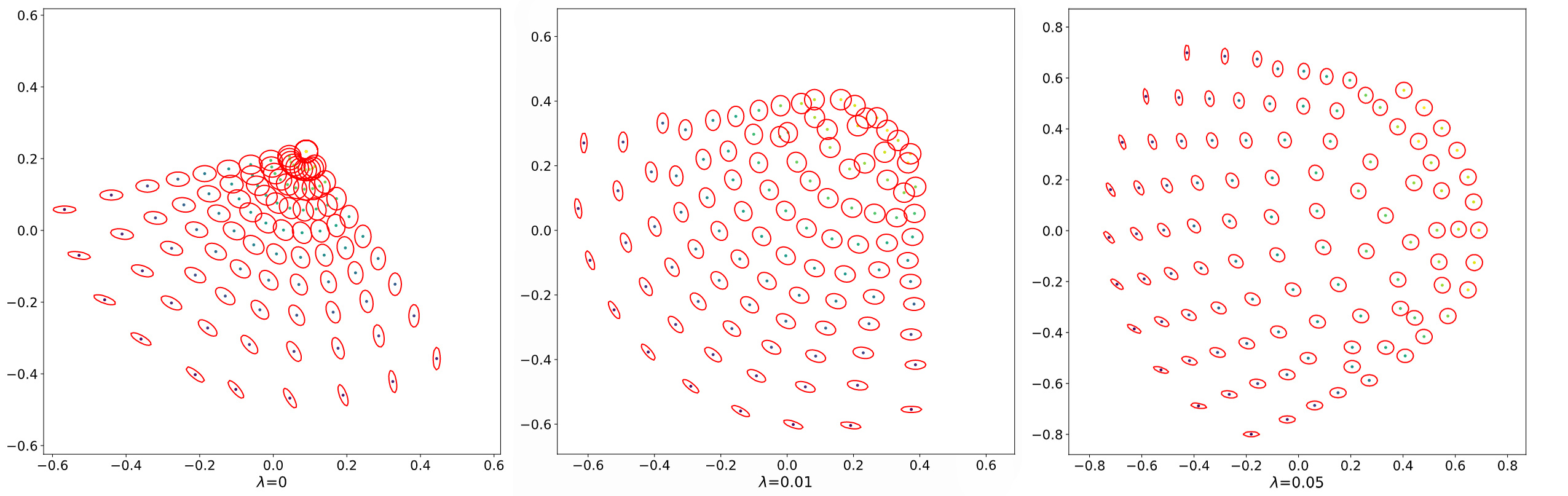}
\end{center}
 \caption{Multidimensional scaling of two-variate normal distributions. The pairwise dissimilarities are given by the square root of the entropy-regularized optimal transport cost $ \tilde{C}_{\lambda}$ for three different regularization parameters $\lambda=0,0.01,0.05$. Each~ellipse in the figure represents a contour of the density function $\{\mathcal{N}(0,\Sigma_{r,k})\}$.}
 \label{fig:embd}
\end{figure}


The following corollary states that if we allow orthogonal transformations of two multivariate normal distributions with fixed covariance matrices, then the minimum and maximum of $C_\lambda$ are attained when $\Sigma_1$ and $\Sigma_2$ are diagonalizable by the same orthogonal matrix or, equivalently, when the ellipsoidal contours of the two density functions are aligned with the same orthogonal axes.

\begin{Corollary} \label{cor:1}
With the same settings as in Theorem \ref{main}, fix $\mu_1$, $\mu_2$, $\Sigma_1$, and all eigenvalues of $\Sigma_2$. When $\Sigma_1$ is diagonalized as $\Sigma_1=\Gamma^{\sm} \Lambda_1^\downarrow \Gamma$, where $\Lambda_1^\downarrow$ is the diagonal matrix of the eigenvalues of $\Sigma_1$ in descending order and $\Gamma$ is an orthogonal matrix, 
\begin{enumerate}
\item[(i)]
$C_\lambda(P,Q)$ is minimized by
$\Sigma_2=\Gamma^{\sm} \Lambda_2^\downarrow \Gamma$ and
\item[(ii)]
$C_\lambda(P,Q)$ is maximized by
$\Sigma_2=\Gamma^{\sm} \Lambda_2^\uparrow \Gamma$,
\end{enumerate}
where
$\Lambda_2^\downarrow$ and $\Lambda_2^\uparrow$ are the diagonal matrices of the eigenvalues of $\Sigma_2$ in descending and ascending order, respectively.
Therefore, neither the minimizer, nor the maximizer depend on the choice of $\lambda$.
\end{Corollary}

\begin{proof}
Because $\mu_1$, $\mu_2$, $\Sigma_1$, and all eigenvalues of $\Sigma_2$ are fixed,
\begingroup\makeatletter\def\f@size{8.8}\check@mathfonts
\def\maketag@@@#1{\hbox{\m@th\normalsize\normalfont#1}}%
\begin{align}
C_{\lambda}(P,Q)&= -2{\rm tr}\left((\Sigma_1^{1/2}\Sigma_2\Sigma_1^{1/2}+\lambda^2 I)^{1/2}\right)
-\frac{\lambda}{2}\log|(\Sigma_1^{1/2}\Sigma_2\Sigma_1^{1/2}+\lambda^2 I)^{1/2}-\lambda I|+({\rm constant})\\
&= \sum_{i=1}^n -2(\nu_i+\lambda^2)^{1/2} -\frac{\lambda}{2}\log\{(\nu_i+\lambda^2)^{1/2}-\lambda\}+({\rm constant})\\
&=\sum_{i=1}^n g_\lambda(\log(\nu_i))+({\rm constant})
\end{align}
\endgroup
where $\nu_1\leq \dots \leq \nu_n$ are the eigenvalues of $\Sigma_1^{1/2}\Sigma_2\Sigma_1^{1/2}$ and: 
\begin{equation}
g_\lambda(x):= -2(e^x+\lambda^2)^{1/2} -\frac{\lambda}{2}\log\{(e^x+\lambda^2)^{1/2}-\lambda\}.
\end{equation}

Note that
$g_\lambda(x)$ is a concave function, because:
\begin{equation}
g_\lambda''(x)= -\frac{e^x(4e^x+7\lambda^2)}{8(e^x+\lambda^2)^{3/2}}<0.
\end{equation}

Let $\nu_1^{\downarrow\downarrow}\leq \dots \leq\nu_n^{\downarrow\downarrow}$ and 
$\nu_1^{\downarrow\uparrow}\leq \dots \leq\nu_n^{\downarrow\uparrow}$ be the eigenvalues of
$\Lambda_1^\downarrow \Lambda_2^\downarrow$ and $\Lambda_1^\downarrow \Lambda_2^\uparrow$, respectively.
By Exercise 6.5.3 of \cite{posi} or Theorem 6.13 and Corollary 6.14 of \cite{hiai2014}, 
\begin{align}\label{eq:cor1}
(\log(\nu_i^{\downarrow\uparrow})) \prec (\log(\nu_i)) \prec (\log(\nu_i^{\downarrow\downarrow})), 
\end{align}

Here, for $(a_i), (b_i) \in \mathbb{R}^n$ such that $a_1\geq\dots\geq a_n$ and $b_1\geq\dots\geq b_n$, 
$(a_i) \prec (b_i)$ means: 
\begin{equation}
\sum_{i=1}^k a_i \leq \sum_{i=1}^k b_i \mbox{~~for~~} k=1,\dots, n-1, \mbox{~~and~~}
\sum_{i=1}^n a_i = \sum_{i=1}^n b_i
\end{equation}
and $(a_i)$ is said to be majorized by $(b_i)$.
Because $g_\lambda(x)$ is concave,
\begin{equation}
g_\lambda(\log(\nu_i^{\downarrow\uparrow})) \prec^{\rm w} g_\lambda (\log(\nu_i)) \prec^{\rm w} g_\lambda(\log(\nu_i^{\downarrow\downarrow})), 
\end{equation}
where $\prec^{\rm w}$ represents weak supermajorization, i.e., 
$(a_i) \prec^{\rm w} (b_i)$ means: 
\begin{equation}
\sum_{i=k}^n a_i \geq \sum_{i=k}^n b_i \mbox{~~for~~} k=1,\dots, n
\end{equation}
(see Theorem 5.A.1 of \cite{marshall2011}, for example).
Therefore,
\begin{equation}
\sum_{i=1}^n g_\lambda(\log(\nu_i^{\downarrow\uparrow}))
\geq \sum_{i=1}^n g_\lambda (\log(\nu_i))
\geq \sum_{i=1}^n g_\lambda(\log(\nu_i^{\downarrow\downarrow})).
\end{equation}

As in Case (i) (or (ii)), the eigenvalues of $\Sigma_1^{1/2}\Sigma_2\Sigma_1^{1/2}$ correspond to the eigenvalues of $\Lambda_1^{\downarrow}\Lambda_2^{\downarrow}$ (or $\Lambda_1^{\downarrow}\Lambda_2^{\uparrow}$, respectively), the corollary follows.
\end{proof}

Note that a special case of Corollary \ref{cor:1} for the ordinary Wasserstein metric ($\lambda=0$)
has been studied in the context of fidelity and the Bures distance in quantum information theory. See Lemma 3 of \cite{markham2008}. Their proof is not directly applicable to our generalized result; thus, we used another approach to prove it.

\section{Extension to Tsallis Entropy Regularization}
\label{Extension to Tsallis entropy regularization}
In this section, we consider a generalization of entropy-regularized optimal transport.
We now focus on the Tsallis entropy \cite{tsallis}, which is a generalization of the Shannon entropy and appears in nonequilibrium statistical mechanics. We show that the optimal coupling of Tsallis entropy-regularized optimal transport between two $q$-normal distributions is also a $q$-normal distribution. We start by recalling the definition of the $q$-exponential function and $q$-logarithmic function based on \cite{tsallis}.
\begin{Definition}
Let $q$ be a real parameter, and let $u>0$. The $q$-logarithmic function is defined as:
\begin{equation}
\log_{q}(u):=\begin{cases}
\frac{1}{1-q}(u^{1-q}-1)\quad\text{if}\ \ q\neq1,\\ 
\log(u) \quad \text{if}\ \ q=1
\end{cases}
\end{equation}
and the $q$-exponential function is defined as:
\begin{equation}
\exp_{q}(u):=
\begin{cases}
[1+(1-q)u]_{+}^{\frac{1}{1-q}} \quad \text{if} \ \ q\neq1,\\
\exp(u)\quad\text{if}\ \ q=1
\end{cases}
\end{equation}
\end{Definition}
\begin{Definition}
Let $q<1$ or $1<q<1+\frac{2}{n}$; an $n$-variate $q$-normal distribution is defined by two parameters: $\mu\in\mathbb{R}^{n}$ and a positive definite matrix $\Sigma$, and its density function is:
\begin{equation}
f(x):=\frac{1}{C_{q}(\Sigma)}\exp_{q}\left(-(x-\mu)^{\sm}\Sigma^{-1}(x-\mu)\right),
\label{elip}
\end{equation}
where $C_{q}(\Sigma)$ is a normalizing constant. $\mu$ and $\Sigma$ are called the location vector and scale matrix,~ respectively.
\end{Definition}
In the following, we write the multivariate $q$-normal distribution $\mathcal{N}_{q}(\mu,\Sigma)$.
We note that the property of the $q$-normal distribution changes in accordance with $q$. 
The $q$-normal distribution has an unbounded support for $1<q<\frac{2}{n}$ and a bounded support for $q<1$.
The second moment exists for $q<1+\frac{2}{n+2}$, and the covariance becomes $\frac{1}{2+(n+2)(1-q)}\Sigma$. We remark that each $n$-variate $\left(1+\frac{2}{\nu+n}\right)$-normal distribution is equivalent to an $n$-variate $t$-distribution with $\nu$ degrees of freedom,
\begin{equation}
\frac{\Gamma[(\nu+n) / 2]}{\Gamma(\nu / 2) \nu^{n / 2} \pi^{n / 2}|{\Sigma}|^{1 / 2}}\left[1+\frac{1}{\nu}({x}-{\mu})^{T} {\Sigma}^{-1}({x}-{\mu})\right]^{-(\nu+n) / 2},
\end{equation}
 for $1<q<1+\frac{2}{n+2}$ and an $n$-variate normal distribution for $q\downarrow 1$.
\begin{Definition}
Let $p$ be a probability density function. The Tsallis entropy is defined as: 
\begin{equation}
S_{q}(p):=\int p(x)\log_{q}\frac{1}{p(x)}dx=\frac{1}{q-1}\left(1-\int p(x)^{q}dx\right).
\end{equation}
\end{Definition}

Then, the Tsallis entropy-regularized optimal transport is defined as:
\begin{align}
\mathrm{minimize}\ &\int\|x-y\|^{2}\pi(x,y)dxdy-2\lambda S_{q}(\pi) \\
\mathrm{subject\ to}\ \ &\int\pi(x,y)dx=q(y)\ \text{for}\ \forall\ y\in\mathbb{R}^n,\nonumber\\
&\int\pi(x,y)dy=p(x)\ \text{for}\ \forall\ x\in\mathbb{R}^n\ . 
\label{siki:constraint_daen}
\end{align}

The following lemma is a generalization of the maximum entropy principle for the Shannon entropy shown in Section 2 of \cite{max_tsallis}.
\begin{Lemma}\label{mep_tsa}
Let $P$ be a centered $n$-dimensional probability measure with a fixed covariance matrix $\Sigma$; the maximizer of the Renyi $\alpha$-entropy: 
\begin{equation}
\frac{1}{1-\alpha} \log \int f(x)^{\alpha} \mathrm{d}x
\end{equation}
under the constraint is $\mathcal{N}_{2-\alpha}(\mathbf{0},((n+2)\alpha-n)\Sigma)$ for $\frac{n}{n+2}<\alpha<1$.
\end{Lemma}
We note that the maximizers of the Renyi $\alpha$-entropy and the Tsallis entropy with $q=\alpha$ coincide; thus, the above lemma also holds for the Tsallis entropy. This is mentioned, for~example, in Section 9 of \cite{naudts}.

To prove Theorem \ref{thm:tsallis}, we use the following property of multivariate $t$-distributions, which is summarized in Chapter 1 of \cite{kotz}.
\begin{Lemma}
Let $X$ be a random vector following an $n$-variate $t$-distribution with degree of freedom $\nu$. Considering a partition of the mean vector $\mu$ and scale matrix $\Sigma$, such as:
\begin{equation}\label{tmarg}
{X}=\left(\begin{array}{l}{X}_{1} \\ {X}_{2}\end{array}\right), \quad{\mu}=\left(\begin{array}{l}{\mu}_{1} \\ {\mu}_{2}\end{array}\right)
,\quad \Sigma=\left(\begin{array}{ll}\Sigma_{11} & \Sigma_{12} \\ \Sigma_{21} & \Sigma_{22}\end{array}\right),
\end{equation}
$X_{1}$ follows a $p$-variate $t$-distribution with degree of freedom $\nu$, mean vector $\mu_{1}$, and scale matrix $\Sigma_{11}$, where $p$ is the dimension of $X_{1}$.
\end{Lemma}
Recalling the correspondence of the parameter of the multivariate $q$-normal distribution and the degree of freedom of the multivariate $t$-distribution $q=1+\frac{2}{\nu+n}$, we can obtain the following corollary.
\begin{Corollary}\label{q-free}
Let $X$ be a random vector following an $n$-variate $q$-normal distribution for $1<q<1+\frac{2}{n+2}$. Consider a partition of the mean vector $\mu$ and scale matrix $\Sigma$ in the same way as in \eqref{tmarg}.
Then, $X_{1}$ follows a $p$-variate $\left(1+\frac{2(q-1)}{2-(n-p)(q-1)}\right)$-normal distribution with mean vector $\mu_{1}$ and scale matrix $\Sigma_{11}$, where $p$ is the dimension of $X_{1}$.
\end{Corollary}
\begin{Theorem}\label{thm:tsallis}
Let $P\sim\mathcal{N}_{q}({\mu_1},\Sigma_1),Q\sim\mathcal{N}_{q}({\mu_2},\Sigma_2)$ be $n$-variate $q$-normal distributions for $1<q<1+\frac{2}{n+2}$ and $\tilde{q}=-\frac{2(q-1)}{ 2-n(q-1)}$; consider the Tsallis entropy-regularized optimal transport:
\begin{equation}
C_{\lambda}(P,Q)=\inf_{\pi\in\Pi(P,Q)}\int_{\mathbb{R}^{n}\times \mathbb{R}^{n}}\|{x}-{y}\|^2\pi({x},{y})d{x}d{y}-2\lambda S_{1+\tilde{q}}(\pi).
\label{siki:daen_frac}
\end{equation}

Then, there exists a unique $\tilde{\lambda}=\tilde{\lambda}(q,\Sigma_{1},\Sigma_{2},\lambda)\in\mathbb{R}_{+}$ such that the optimal coupling $\pi$ of the entropy-regularized optimal transport
is expressed as:
\begin{equation}
\pi\sim\mathcal{N}_{1-\tilde{q}}\left(\begin{pmatrix}
{\mu}_1\\
{\mu}_2 \\
\end{pmatrix},
\begin{pmatrix}
\Sigma_1 &\Sigma_{\tilde{\lambda}} \\
\Sigma_{\tilde{\lambda}}^{\sm} & \Sigma_2 \\
\end{pmatrix}\right),
\end{equation}
where:
\begin{equation}
\Sigma_{\tilde{\lambda}}:=\Sigma_1^{1/2}(\Sigma_1^{1/2}\Sigma_2\Sigma_1^{1/2}+{\tilde{\lambda}}^2I)^{1/2}\Sigma_1^{-1/2}-{\tilde{\lambda}}I.
\label{daensigma}
\end{equation}
\end{Theorem}
\begin{proof}
The proof proceeds in a similar way as in Theorem \ref{main}. Let $\alpha\in L(P)$ and $\beta\in L(Q)$ be the Lagrangian multipliers. Then, the Lagrangian function $L(\pi,\alpha,\beta)$ of \eqref{siki:constraint_daen} is defined~as:
\begin{align}
L(\pi,\alpha,\beta):=&\int\|x-y\|^2\pi(x,y)dxdy-2\lambda \left\{\frac{1}{\tilde{q}}\left(1-\int\pi(x,y)^{1+\tilde{q}}dxdy\right)\right\}\nonumber\\
&-\int\alpha(x)\pi(x,y)dxdy+\int\alpha(x)p(x)dx\nonumber\\
&-\int\beta(y)\pi(x,y)dxdy+\int\beta(y)q(y)dy\label{lag}
\end{align}
and the extremum of the Tsallis entropy-regularized optimal transport is obtained by the functional derivative with respect to $\pi$,
\begin{equation}
\pi(x,y)=\left(\frac{\tilde{q}}{{2(\tilde{q}+1)\lambda}}\left(-\alpha({x})-\beta({y})+\|{x}-{y}\|^2\right)\right)^\frac{1}{\tilde{q}}.
\end{equation}

Here, $\alpha$ and $\beta$ are quadratic polynomials by Lemma \ref{mep_tsa}.
To separate the normalizing constant, we introduce a constant $c\in\mathbb{R}_{+}$, and $\pi$ can be written as:
\begin{equation}
\pi(x,y)=c^\frac{1}{\tilde{q}}\left(\tilde{\alpha}(x)+\tilde{\beta}({y})+\frac{\tilde{q}\|{x}-{y}\|^2}{2c(\tilde{q}+1)\lambda}\right)^\frac{1}{\tilde{q}},
\end{equation}
with quadratic functions $\tilde{\alpha}(x)$ and $\tilde{\beta}(y)$.

Let $\tilde{\lambda} = \frac{c(\tilde{q}+1)\lambda}{{\tilde{q}}}>0$. Then, by the same argument as in the proof of Theorem \ref{main} and using Corollary \ref{q-free}, we obtain the scale matrix of $\pi$ as:
\begin{equation}
\Sigma=\begin{pmatrix}
\Sigma_1 &\Sigma_{\tilde{\lambda}} \\
\Sigma_{\tilde{\lambda}}^{\sm} & \Sigma_2 \\
\end{pmatrix},
\end{equation}
where:
\begin{equation}
\Sigma_{\tilde{\lambda}}=\Sigma_1^{1/2}(\Sigma_1^{1/2}\Sigma_2\Sigma_1^{1/2}+\tilde{\lambda}^2I)^{1/2}\Sigma_1^{-1/2}-{\tilde{\lambda}}I.
\end{equation}

Let $z=(x^{\sm},y^{\sm})^{\sm}$ and $K_{\tilde{q}}=\int(1+z^{\sm}z)^\frac{1}{\tilde{q}}dz$; $\pi$ can be written as:
\begin{equation}
\pi(x,y) =\frac{1}{K_{\tilde{q}}|\Sigma|} (1+z^{\sm} \Sigma^{-1}z)^{\frac{1}{\tilde{q}}}.
\end{equation}

The constant $c$ is determined by:
\begin{equation}
\frac{1}{K_{\tilde{q}}|\Sigma|}=c^\frac{1}{\tilde{q}}.
\label{unic}
\end{equation}

We will show that the above equation has a unique solution. 
Let $\{\tau\}_{i=1}^{n}$ be the eigenvalues of $(\Sigma_{1}^{1/2}\Sigma_{2}\Sigma_{1}^{1/2})^{1/2}$; $|\Sigma|$ can be expressed as 
$\prod_{i=1}^{2n}2\tilde{\lambda}(\sqrt{\tau_{i}^{2}+\tilde{\lambda}^{2}}-\tilde{\lambda})$.
We~consider: 
\begin{align}
f(c) &= \log(c^{\frac{1}{\tilde{q}}}{K_{\tilde{q}}|\Sigma|})\\
&= \frac{1}{\tilde{q}}\log c +\sum_{i=1}^{2n} \log(\sqrt{\tau_{i}^{2}+\tilde{\lambda}^{2}}-\tilde{\lambda}) +2n\log(2\tilde{\lambda})+\log K_{\tilde{q}}.
\end{align}

Because $\tilde{q}<0$, $f(c)$ is a monotonic decreasing function, and $\lim_{c\downarrow 0}f(c)=\infty$, $\lim_{c\to \infty}f(c)=-\infty$, \eqref{unic} has a unique positive solution, and $\tilde{\lambda}$ is determined uniquely.
\end{proof}
\section{Entropy-Regularized Kantorovich Estimator}
\label{Entropy-Regularized Kantorovich Estimator}
Many estimators are defined by minimizing the divergence or distance $\rho$ between probability measures, that is $\arg\min_{\mu}\rho(\mu,\nu)$ for a fixed $\nu$. When $\rho$ is the Kullback--Leibler divergence, the estimator corresponds to the maximum likelihood estimator. When $\rho$ is the Wasserstein distance, the following estimator is called the minimum Kantorovich estimator, according to \cite{copcuturi}.
In this section, we consider a probability measure $Q^{*}$ that minimizes $C_{\lambda}(P,Q)$ for a fixed $P$ over $\mathcal{P}_{2}(\mathbb{R}^{n})$, the set of all probability measures on $\mathbb{R}^{n}$ with finite second moment that are absolutely continuous with respect to the Lebesgue measure.
In other words, we define the entropy-regularized Kantorovich estimator $\arg\min_{Q\in \mathcal{P}_{2}(\mathbb{R}^{n})} C_{\lambda}(P,Q).$
The entropy-regularized Kantorovich estimator for discrete probability measures was studied in \cite{amari_new}, Theorem 2.
We obtain the entropy-regularized Kantorovich estimator for continuous probability measures in the following theorem: 
\begin{Theorem} \label{thm:mke}
For a fixed $P \in \mathcal{P}_{2}(\mathbb{R}^{n})$,
\begin{equation}
Q^{*}=\arg\min_{Q\in \mathcal{P}_{2}(\mathbb{R}^{n})} C_{\lambda}(P,Q)
\end{equation}
exists, and its density function can be written as:
\begin{equation}
dQ^{*}=dP\star\phi_{\lambda},
\end{equation}
where $\phi_{\lambda}(x)$ is a density function of $\mathcal{N}(0,\frac\lambda2I)$, and $\star$ denotes the convolution operator.
\end{Theorem}
We use the dual problem of the entropy-regularized optimal transport to prove \mbox{Theorem \ref{thm:mke}} (for details, see Proposition 2.1 of \cite{stosink} or Section 3 of \cite{clason2021entropic}).
\begin{Lemma}
\label{thmdual}
The dual problem of entropy-regularized optimal transport can be written as:
\begin{align}
\mathcal{A}_{\lambda}(P,Q)=\sup_{\substack{\alpha\in L_{1}(P)\\\beta\in L_{1}(Q)}}&\int \alpha(x)p(x)dx+\int \beta(y)q(y)dy\nonumber\\
&-\lambda\int\exp\left\{\frac{\alpha(x)+\beta(y)-\|x-y\|^{2}}{\lambda}\right\}dxdy\label{dualthm}.
\end{align}

Moreover, $\mathcal{A}_{\lambda}(P,Q)=C_{\lambda}(P,Q)$ holds.
\end{Lemma}
Now, we prove Theorem \ref{thm:mke}.
\begin{proof}

Let $Q^{*}$ be the minimizer of $\min_{Q}C_{\lambda}(P,Q)$. Applying Lemma \ref{thmdual}, there exist \mbox{$\alpha^{*}\in L_{1}(P)$} and $\beta^{*}\in L_{1}(Q^{*})$ such that: 
\begin{align}
C_{\lambda}(P,Q^{*})=\mathcal{A}_{\lambda}(P,Q^{*})&=\int \alpha^{*}(x)p(x)dx+\int \beta^{*}(y)q^{*}(y)dy\nonumber\\
&-\lambda\int\exp\left\{\frac{\alpha^{*}(x)+\beta^{*}(y)-\|x-y\|^{2}}{\lambda}\right\}dxdy.
\label{mindual}
\end{align}

Now, $\mathcal{A}_{\lambda}(P,Q^{*})$ is the minimum value of $\mathcal{A}_{\lambda}$, such that the variation $\delta\mathcal{A}_{\lambda}(P,Q^{*})$ is always zero. Then,
\begin{equation}
\delta\mathcal{A}_{\lambda}(P,Q^{*})=\int \beta^{*}(y)\delta q^{*}(y)dy=0\Rightarrow \beta^{*}\equiv0
\end{equation}
holds, and the optimal coupling of $P,Q$ can be written as:
\begin{align}
\pi^{*}(x,y)&=\exp\left\{\frac{\alpha^{*}(x)+\beta^{*}(y)}{\lambda}-\frac{\|x-y\|^{2}}{\lambda}\right\}\\
&=\exp\left\{\frac{\alpha^{*}(x)}{\lambda}\right\}\exp\left\{-\frac{\|x-y\|^{2}}{\lambda}\right\}.
\end{align}

Moreover, we can obtain a closed-form of $\alpha^{*}(x)$ as follows from the equation $\int \pi(x,y) dy$ $=p(x)$:
\begin{equation}
\frac{\alpha^{*}(x)}{\lambda}=\log p(x)-\log\int\exp\left\{-\frac{\|x-y\|^{2}}{\lambda}\right\}dy=\log p(x)-\frac n2\log(\pi\lambda).
\end{equation}

Then, by calculating the marginal distribution of $\pi(x,y)$ with respect to $x$, \mbox{we can obtain}: 
\begin{equation}
q^{*}(y)=\int\frac{1}{(\pi{\lambda})^{\frac n2}}\exp\left\{-\frac{\|x-y\|^{2}}{\lambda}\right\}p(x)dx=(p\star\phi_{\lambda})(y).
\label{opq}
\end{equation}

Therefore, we conclude that a probability measure $Q$ that minimizes $C_{\lambda}(P,Q)$ is expressed as \eqref{opq}.
\end{proof}

It should be noted that when $P$ in Theorem \ref{thm:mke} are multivariate normal distributions, $Q^*$ and $P$ are simultaneously diagonalizable by a direct consequence of the theorem. This~is consistent with the result of Corollary \ref{cor:1}(1) for minimization when all eigenvalues are fixed.

We can determine that the entropy-regularized Kantorovich estimator is a measure convolved with an isotropic multivariate normal distribution scaled by the regularization parameter $\lambda$. This is similar to the idea of prior distributions in the context of Bayesian inference. Applying Theorem \ref{thm:mke}, the entropy-regularized Kantorovich estimator of the multivariate normal distribution $ \mathcal{N}(\mu,\Sigma)$ is $ \mathcal{N}(\mu,\Sigma+\frac{\lambda}{2}I)$.

\section{Numerical Experiments}
\label{Numerical Experiments}
In this section, we introduce experiments that show the statistical efficiency of entropy regularization in Gaussian settings. We consider two different setups, estimating covariance matrices (Section \ref{Estimation of Covariance Matrices}) and the entropy-regularized Wasserstein barycenter (Section \ref{Estimation of the Wasserstein Barycenter}). To obtain the entropy-regularized Wasserstein barycenter, we adopt the Newton--Schulz method and a manifold optimization method, which are explained in Sections \ref{Section6.3} and \ref{Approximate the Matrix Square Root}, respectively.
\subsection{Estimation of Covariance Matrices}
\label{Estimation of Covariance Matrices}
We provide a covariance estimation method based on entropy-regularized optimal transport. Let $P=\mathcal{N}(\mu,\Sigma)$ be an $n$-variate normal distribution. We define an entropy-regularized Kantorovich estimator $\hat{P}_{\lambda}$, that is, 
\begin{equation}
\hat{P}_{\lambda}=\arg\min_{Q} C_\lambda(P,Q).
\end{equation}

We generate some samples from $\mathcal{N}(\mu,\Sigma)$ and estimate the mean and covariance matrix.
We compare the maximum likelihood estimator $\hat{P}_{\mathrm{MLE}}=\mathcal{N}(\hat{\mu}_{\mathrm{MLE}},\hat{\Sigma}_{\mathrm{MLE}})$ and $\hat{P}_{\lambda}$ with respect to the prediction error:
\begin{equation}
\mathrm{KL}(P,\hat{P}_{\mathrm{MLE}}),\ \ \mathrm{KL}(P,\hat{P}_{\lambda}).
\end{equation}

In our experiment, the dimension $n$ is set to $5,15,30$, and the sample size is set to $60,120$. The experiment proceeds as follows.
\begin{enumerate}
 \item Obtain a random sample of size 60 (or 120) from $\mathcal{N}(0,\Sigma)$ and its sample covariance matrix $\hat{\Sigma}$.
 \item Obtain the entropy-regularized minimum Kantorovich estimator of $\hat{\Sigma}$ obtained in Step 1.
 \item Compute the prediction error between $\Sigma$ and the entropy-regularized minimum Kantorovich estimator of $\hat{\Sigma}$
 \item Repeat the above steps 1000 times and obtain a confidence interval of the prediction~error.
\end{enumerate}

Table \ref{tab:mle} shows the average prediction error of the MLE and entropy-regularized Kantorovich estimator of covariance matrices from 60 samples from an n-variate normal distribution with the 95\% confidential interval. We can see that the prediction error is smaller than the maximum likelihood estimator under adequately small $\lambda$ for $n=15,30$, but not for $n=5$. Moreover, the decrease in the prediction error is larger for $n=30$ than for $n=15$, which indicates that the entropy regularization is effective in a high dimension.
On the other hand, Table \ref{tab:mle2} shows in all cases that the decreases in the prediction error are more moderate than Table \ref{tab:mle}. We can see that this is due to the increase in the sample size. 
Then, we can conclude that the entropy regularization is effective in a high-dimensional setting with a small sample size.

\begin{table}[H]
\caption{Average prediction error of the MLE and entropy-regularized Kantorovich estimator of covariance matrices from 60 samples from an $n$-variate normal distribution with the 95\% confidential~interval.}
\begin{center}
\begin{tabular*}{\hsize}{@{}@{\extracolsep{\fill}}lcccc@{}}
\hline
\boldmath{$\lambda$}&\boldmath{$\mathrm{KL}(P,\hat{P}_{\mathrm{W}}),n=5$}&\boldmath{$\mathrm{KL}(P,\hat{P}_{\mathrm{W}}),n=15$}&\boldmath{$\mathrm{KL}(P,\hat{P}_{\mathrm{W}}),n=30$} \\ \hline
0(MLE) &0.062 $\pm\ 0.005$&1.346 $\pm\ 0.022$& 10.69 $\pm\ 0.112$ \\
0.01 &0.051 $\pm\ 0.005$&1.242 $\pm\ 0.021$& 8.973 $\pm \ 0.087$ \\
0.1 &0.104 $\pm\ 0.004$&0.841 $\pm\ 0.013$& 4.180 $\pm \ 0.033$ \\
0.5 &0.647 $\pm\ 0.003$&0.931 $\pm\ 0.007$& 3.093 $\pm \ 0.010$ \\
1.0 &1.166 $\pm\ 0.003$&1.670 $\pm\ 0.006$& 5.075 $\pm \ 0.009$ \\ \hline
\end{tabular*}
\end{center}
\label{tab:mle}
\end{table}

\vspace{-12pt}

\begin{table}[H]
\caption{Average prediction error of the MLE and entropy-regularized Kantorovich estimator of covariance matrices from 120 samples from an $n$-variate normal distribution with the 95\% confidential~interval.}
\begin{tabular*}{\hsize}{@{}@{\extracolsep{\fill}}lcccc@{}}
\hline
\boldmath{$\lambda$}&\boldmath{$\mathrm{KL}(P,\hat{P}_{\mathrm{W}}),n=5$}&\boldmath{$\mathrm{KL}(P,\hat{P}_{\mathrm{W}}),n=15$}&\boldmath{$\mathrm{KL}(P,\hat{P}_{\mathrm{W}}),n=30$} \\ \hline
0(MLE) &0.024 $\pm\ 0.002$&0.490 $\pm\ 0.007$& 2.810 $\pm\ 0.021$ \\
0.01 &0.020 $\pm\ 0.002$&0.459 $\pm\ 0.006$& 2.528 $\pm \ 0.018$ \\
0.1 &0.101 $\pm\ 0.002$&0.397 $\pm\ 0.005$& 1.700 $\pm \ 0.001$ \\
0.5 &0.659 $\pm\ 0.002$&0.875 $\pm\ 0.004$& 2.833 $\pm \ 0.005$ \\
1.0 &1.180 $\pm\ 0.002$&1.730 $\pm\ 0.004$& 5.124 $\pm \ 0.005$ \\ \hline
\end{tabular*}
\label{tab:mle2}
\end{table}

\subsection{Estimation of the Wasserstein Barycenter}
\label{Estimation of the Wasserstein Barycenter}
A barycenter with respect to the Wasserstein distance is definable \cite{barycenters} and is widely used for image interpolation and 3D object interpolation tasks with entropy regularization~\cite{amari_new,conv}.
\begin{Definition}
Let $\{Q_{i}\}_{i=1}^{m}$ be a set of probability measures in $\mathcal{P}(\mathbb{R}^{n})$. The barycenter with respect to $C_{\lambda}$ (entropy-regularized Wasserstein barycenter) is defined as:
\begin{equation}
\arg\min_{P\in\mathcal{P}(\mathbb{R}^n)}\sum_{i=1}^{m} C_{\lambda}(P,Q_{i}).
\label{eq:barycenter}
\end{equation}
\end{Definition}
Now, we restrict $P$ and $\{Q_{i}\}_{i=1}^{m}$ to be multivariate normal distributions and apply our theorem to illustrate the effect of entropy regularization.

The experiment proceeds as follows. The dimensionality and the sample size were set the same as in the experiments in Section \ref{Estimation of Covariance Matrices}.
\begin{enumerate}
 \item Obtain a random sample of size 60 (or 120) from $\mathcal{N}(0,\Sigma)$ and its sample covariance matrix $\hat{\Sigma}$.
 \item Repeat Step 1 three times, and obtain $\{\hat{\Sigma}\}_{i=1}^{3}$.
 \item Obtain the barycenter of $\{\hat{\Sigma}_i\}_{i=1}^{3}$.
 \item Compute the prediction error between $\Sigma$ and the barycenter obtained in step 3.
 \item Repeat the above steps 100 times and obtain a confidence interval of the prediction~error.
 \end{enumerate}

We show the results for several values of the regularization parameter $\lambda$ in \mbox{Tables \ref{tab:bary} and \ref{tab:bary2}}. A decrease in the prediction error can be seen in Table \ref{tab:bary} for $n=30$, as well as \mbox{Tables \ref{tab:mle} and \ref{tab:mle2}}. However, because the computation of the entropy-regularized Wasserstein barycenter uses more data than that of the minimum Kantorovich estimator, the decrease in the prediction error is mild. The entropy-regularized Kantorovich estimator is a special case of the entropy-regularized Wasserstein barycenter \eqref{eq:barycenter} for $m=1$. Our experiments show that the appropriate range of $\lambda$ to decrease the prediction error depends on $m$ and becomes narrow as $m$ increases. 
In addition, we note that there is a small decrease in the prediction error in Table \ref{tab:bary2} for $n=30$.

\begin{table}[H]
\caption{Average prediction error of the entropy-regularized barycenter with the 95\% confidential interval (random sample of size 60).}
\begin{tabular*}{\hsize}{@{}@{\extracolsep{\fill}}lcccc@{}}
\hline
\boldmath{$\lambda$}&\boldmath{$\mathrm{KL}(P,\hat{P}_{\mathrm{W}}),n=5$}&\boldmath{$\mathrm{KL}(P,\hat{P}_{\mathrm{W}}),n=15$}&\boldmath{$\mathrm{KL}(P,\hat{P}_{\mathrm{W}}),n=30$} \\ \hline
0 &0.455 $\pm\ 0.395$&1.318 $\pm\ 0.006$& 4.875 $\pm\ 0.035$ \\
0.001 &0.429 $\pm\ 0.396$&1.318 $\pm\ 0.004$& 4.887 $\pm \ 0.036$ \\
0.01 &0.434 $\pm\ 0.270$&1.344 $\pm\ 0.006$& 4.551 $\pm \ 0.164$ \\
0.025 &0.780 $\pm\ 0.223$&1.456 $\pm\ 0.064$& 5.710 $\pm \ 0.536$ \\
0.005 &1.047 $\pm\ 0.029$&1.537 $\pm\ 0.064$& 7.570 $\pm \ 0.772$ \\ \hline
\end{tabular*}
\label{tab:bary}
\end{table}

\vspace{-12pt}
\begin{table}[H]
\caption{Average prediction error of the entropy-regularized barycenter with the 95\% confidential interval (random sample of size 120).}
\begin{tabular*}{\hsize}{@{}@{\extracolsep{\fill}}lcccc@{}}
\hline
\boldmath{$\lambda$}&\boldmath{$\mathrm{KL}(P,\hat{P}_{\mathrm{W}}),n=5$}&\boldmath{$\mathrm{KL}(P,\hat{P}_{\mathrm{W}}),n=15$}&\boldmath{$\mathrm{KL}(P,\hat{P}_{\mathrm{W}}),n=30$} \\ \hline
0 &0.154 $\pm\ 0.600$&1.303 $\pm\ 0.010$& 5.091 $\pm\ 0.035$ \\
0.001 &0.212 $\pm\ 0.070$&1.305 $\pm\ 0.010$& 5.072 $\pm \ 0.037$ \\
0.01 &0.306 $\pm\ 0.046$&1.328 $\pm\ 0.008$& 5.274 $\pm \ 0.252$ \\
0.025 &0.671 $\pm\ 0.028$&1.337 $\pm\ 0.073$& 5.851 $\pm \ 0.424$ \\
0.005 &1.109 $\pm\ 0.063$&1.603 $\pm\ 0.184$& 8.072 $\pm \ 0.725$ \\ \hline
\end{tabular*}
\label{tab:bary2}
\end{table}

\subsection{Gradient Descent on $\mathrm{Sym}_{+}(n)$}
\label{Section6.3}
We use a gradient descent method to compute the entropy-regularized barycenter. Applying the gradient descent method to the loss function defined by the Wasserstein distance was proposed in \cite{elipap}. This idea is extendable to entropy-regularized optimal transport. The detailed algorithm is shown below.
Because $C_{\lambda}(P,Q)$ is a function of a positive definite matrix, we used a manifold gradient descent algorithm on the manifold of positive definite matrices.

We review the manifold gradient descent algorithm used in our numerical experiment.
Let $\mathrm{Sym}_{+}(n)$ be the manifold of $n$-dimensional positive definite matrices. We require a formula for a gradient operator and the inner product of $\mathrm{Sym}_{+}(n)$ in the gradient descent algorithm. In this paper, we use the following inner product from \cite{posi}, Chapter 6. 
For a fixed $X \in \mathrm{int}(\mathrm{Sym}_{+}(n))$, we define an inner product of $\mathrm{Sym}_{+}(n)$ as:
\begin{equation}
\label{inner}
g_{X}(Y, Z)=\operatorname{tr}\left(Y X^{-1} Z X^{-1}\right), \ Y, Z \in \mathrm{Sym}_{+}(n), 
\end{equation}

Equation \eqref{inner} is the best choice in terms of the convergence speed according to \cite{inner}.
Let $f: \mathrm{Sym}_{+}(n)\rightarrow \mathbb{R}$ be a differential matrix function. Then, the induced gradient of $f$ under \eqref{inner} is: 
\begin{equation}
\operatorname{grad}f(X)=X \left(\frac{\partial f(X)}{\partial X} \right)X.
\end{equation}

We consider the updating step after obtaining the gradient of $f$. $\mathrm{grad}f(X)$ is an element of the tangent space, and we have to project it to $\mathrm{Sym}_{+}(n)$. This projection map is called a retraction. It is known that the Riemannian metric $g_{X}$ leads to the following retraction:

\noindent $\exp _{X} x=X\operatorname{Exp}\left(X^{-1}x \right)$, where $\mathrm{Exp}$ is the matrix exponential. Then, the corresponding gradient descent method becomes as shown in Algorithm \ref{alg:grad}.

\subsection{Approximate the Matrix Square Root}
\label{Approximate the Matrix Square Root}
To compute the gradient of the square root of a matrix in the objective function, we~approximate it using the Newton--Schulz method \cite{newton}, which can be implemented by matrix operations as shown in Algorithm \ref{alg:newton}.
It is amenable to automatic differentiation, such that we can easily apply the gradient descent method to our algorithm.

\begin{algorithm}[H] 
\caption{Gradient descent on the manifold of positive definite matrices.}
\label{alg:grad} 
\begin{algorithmic} 
\REQUIRE $f(X)$
\STATE\textbf{initialize} $X$
\WHILE { no convergence}
\STATE $\eta:\ $step size
\STATE $\mathrm{grad}\leftarrow X \left(\frac{\partial f(X)}{\partial X} \right)X$
\STATE $X\leftarrow \exp_{X}(\eta\mathrm{grad})=X\mathrm{Exp}(\eta X^{-1}\mathrm{grad})$
\ENDWHILE
\ENSURE $X$
\end{algorithmic}
\end{algorithm}

\vspace{-10pt}

\begin{algorithm}[H] 
\caption{Newton--Schulz method.} 
\label{alg:newton} 
\begin{algorithmic} 
\REQUIRE $A\in\mathrm{Sym}_{+}(n),\epsilon>0$
\STATE$Y\leftarrow\frac{A}{(1+\epsilon)\|A\|},\ \ Z\leftarrow I$
\WHILE { no convergence}
\STATE $T\leftarrow(3I-ZY)/2$
\STATE $Y\leftarrow YT,\ Z\leftarrow TZ$
\ENDWHILE
\ENSURE $\sqrt{(1+\epsilon)\|A\|}Y$
\end{algorithmic}
\end{algorithm}

\section{Conclusions and Future Work}
In this paper, we studied entropy-regularized optimal transport and derived several result. We summarize these as follows and add notes on future work.
\begin{itemize}
\item We obtain the explicit form of entropy-regularized optimal transport between two multivariate normal distributions and derived Corollaries \ref{cor:mono} and \ref{cor:1}, which clarified the properties of optimal coupling. Furthermore, we demonstrate experimentally how entropy regularization affects the Wasserstein distance, the optimal coupling, and~the geometric structure of multivariate normal distributions. Overall, the properties of optimal coupling were revealed both theoretically and experimentally. We~expect that the explicit formula can be a replacement for the existing methodology using the (nonregularized) Wasserstein distance between normal distributions (for example,~\mbox{\cite{elipap,inception}}).

\item Theorem \ref{thm:tsallis} derives the explicit form of the optimal coupling of the Tsallis entropy-regularized optimal transport between multivariate $q$-normal distributions. The~optimal coupling of the Tsallis entropy-regularized optimal transport between multivariate $q$-normal distributions is also a multivariate $q$-normal distribution, and the obtained result has an analogy to that of the normal distribution. We believe that this result can be extended to other elliptical distribution families.

\item The entropy-regularized Kantorovich estimator of a probability measure in $\mathcal{P}_2(\mathbb{R})$ is the convolution of a multivariate normal distribution and its own density function. Our experiments show that both the entropy-regularized Kantorovich estimator and the Wasserstein barycenter of multivariate normal distributions outperform the maximum likelihood estimator in the prediction error for adequately selected $\lambda$ in a high dimensionality and small sample setting. As future work, we want to show the efficiency of entropy regularization using real data.
\end{itemize}
\section{acknowledgements}
This work was supported by RIKEN AIP and JSPS KAKENHI (JP19K03642, JP19K00912).
\bibliography{cite} 
\bibliographystyle{plain} 

\end{document}